\documentclass[12pt,a4paper,twoside]{article}
\usepackage{a4wide, amssymb, amsmath, amsthm, graphics, comment, xspace, enumerate}
\usepackage{graphicx,multirow}

\usepackage{color, colortbl} 
\usepackage{float} 
\usepackage[margin=2.5cm]{geometry} 
\usepackage{cite}
\usepackage{authblk} 

\theoremstyle{definition}
\newtheorem{definition}{Definition}[section]

\newtheorem{remark}[definition]{Remark}

\theoremstyle{plain}
\newtheorem{thm}[subsection]{Theorem}
\newtheorem{lem}[subsection]{Lemma}

\newcommand{\beq}{\begin{eqnarray}}
\newcommand{\eeq}{\end{eqnarray}}

\newcommand{\beqs}{\begin{eqnarray*}}
\newcommand{\eeqs}{\end{eqnarray*}}

\allowdisplaybreaks

\usepackage{mathtools}

\title{\bf Towards the Solution of an Extremal Problem Concerning the Wiener Polarity Index of Alkanes }

\author{
   \large \bf Sadia Noureen$^{a}$, Akhlaq Ahmad Bhatti$^{a}$, Akbar Ali$^{b}$
}

\affil{ \normalsize
    {$^{a}$Department of Mathematics\\ National University of Computer and Emerging Sciences, Lahore, Pakistan}\\
    E-mail: {\texttt{sadia.tauseef@uog.edu.pk, akhlaq.ahmad@nu.edu.pk}}\\
    {$^{b}$Department of Mathematics, Faculty of Science\\ University of Ha'il, Ha'il, Saudi Arabia}
    \\E-mail: {\texttt{akbarali.maths@gmail.com}}
}

\begin{document}

\maketitle

\begin{abstract}
The Wiener polarity index $W_p$, one of the most studied molecular structure descriptors, was devised by the chemist Harold Wiener for predicting the boiling points of alkanes. The index $W_p$ for chemical trees (chemical graphs representing alkanes) is defined as the number of unordered pairs of vertices at distance 3.
A vertex of a chemical tree with degree at least 3 is called a branching vertex. A segment of a chemical tree $T$ is a path-subtree $S$ whose terminal vertices have degrees different from 2 in $T$ and every internal vertex (if exists) of $S$ has degree 2 in $T$. In this paper, the best possible sharp upper and lower bounds on the Wiener polarity index $W_p$ are derived for the chemical trees of order $n$ with a given number of branching vertices or segments, and the corresponding extremal chemical trees are characterized. As a consequence of the derived results, an open problem concerning the maximal $W_p$ value of chemical trees with a fixed number of segments or branching vertices is solved.

\baselineskip=0.30in
\end{abstract}
%
%
%
%
%
%
%
%

\baselineskip=0.25in
\section[Introduction]{Introduction}

In molecular science one of the important issues is to  model and predict the physicochemical properties of chemical compounds. Many theoretical methods have been developed by different researchers in this regard. Nowadays people are more concerned in one of the methods that involve topological indices \cite{Balban-13}. The use of topological indices has made the life more easier to understand the molecular science.\\

Every chemical compound can be represented by a graph (known as a chemical graph) in which vertices are the atoms of the chemical compound and edges are the bonds. A \textit{topological index} is a numerical value associated with a chemical graph, which remains unchanged under graph isomorphism \cite{Estrada-13}. In 1947, Harold Wiener proposed a topological index while he was working on boiling points of the paraffins \cite{Wiener-47}. He proposed a linear formula in which two parameters, say $W$ and $W_p$, were used, which were later on named as the Wiener index and the Wiener polarity index, respectively. Much attention was not given to $W_p$ until Lukovits and Linert demonstrated the importance of $W_p$ in the  quantitative structure-property relationships (QSPR) in a series of acyclic and cycle-containing hydrocarbons \cite{Luk-98}. Detail about the chemical applications of $W_p$ can be found in the papers\cite{Hosoya-02,Luk-98,Milic-04,Safari-17,Shafi-17a,Shafi-17b,Wiener-47,Tratnik-19,Du-18b}. Considering the importance of this topological index, many researchers have devoted their attention towards it and studied its mathematical properties, for example see the papers \cite{Chen-16,Hua-16,Lei-18,Lei-17,Ma-16,Ma-14,Yue-18,Zhang-16,Du-18a} and references listed therein.\\

Before going further, let us recall some graph theoretical definitions and notations. Let $G=(V(G),E(G))$ be a simple graph with the vertex set $V(G)$ and edge set $E(G)$. The degree of a vertex $u\in V(G)$, denoted by $d_u(G)$ (or simply by $d_u$ if the graph under consideration is clear), is the number of edges incident with $u$. A connected graph of order $n$ and size $n-1$ is called a \textit{tree}.  A tree with maximum degree at most 4 is called a \textit{chemical tree}. A vertex of degree one in a graph is said to be \textit{pendent vertex} and a vertex of degree greater than 2 is called a \textit{branching vertex}. A \textit{segment} of a tree T (see \cite{Dobrynin}) is a path-subtree $S$ whose terminal vertices have degrees different from 2 in $T$ and every internal vertex (if exists) of $S$ has degree 2 in $T$. A sequence $P=v_0v_1\cdots v_k$ of vertices of a tree $T$ is called a \textit{pendent path} (\textit{internal path}, respectively) of length $k$, if each two consecutive vertices in $P$ are adjacent in $T$, one of the two vertices $v_0$, $v_k$ is pendent and the other is branching (both the vertices $v_0$, $v_k$ are branching, respectively), and $d_{v_i}=2$ if $1\le i \le k-1$.\\

Du \textit{et al.}\cite{Du-09} showed that Wiener polarity index $W_p$ of any tree $T$ can be written as:
$$W_{p}(T)= \sum_{uv\in E(T)}{(d_u-1)(d_v-1)}.$$
Here, it should be noted that in case of trees, $W_p$ coincides with the reduced second Zagreb index \cite{Furtula-14,Gutman-14,Shaf-17}. In \cite{Furtula-14,Du-09,Shaf-17} the researchers determined the bounds on $W_p$ for some classes of graphs\,. In \cite{Liu-12}, the authors determined maximum $W_p$ among the class of $n$-vertex trees with the given degree sequence. Deng \textit{et al.} in \cite{Deng-11} determined the maximum and minimum $W_p$ value of the chemical trees on $n$ vertices with $n\ge 4$. For, further references on extremal results on $W_p$ we refer the reader to the references \cite{Ali-18a,Du-18a,Ashrafi-17,Deng-10a,Hou-12}.\\

The main motivation of the present study comes from the recent paper \cite{Noureen-20} where the problem of finding the graphs having maximal $W_p$ value among all chemical trees of order $n$ with a fixed number of segments or branching vertices was attacked and left open. In this paper, we solve this problem.\\

In Sections 2 and 3, the best possible sharp upper and lower bounds on $W_p$ are derived, respectively, for the chemical trees of order $n$ with a given number of branching vertices, and the corresponding extremal trees are characterized. The best possible sharp upper and lower bounds on $W_p$ for chemical trees of order $n$ with a fixed number of segments are established in Sections 4 and 5, and the trees attaining these bounds are also characterized there. In what follows, we define few notions that will be used in the next sections.\\

A tree of order $n$ is called $n$-vertex tree. Let $\mathcal{CT}_{n,k}$ be the class of all $n$-vertex chemical trees with exactly $k$ segments, where $3\le k \le n-2$. Denote by $\mathcal{CT}_{n,b}^*$ the class of all $n$-vertex chemical trees with exactly $b$ branching vertices, where $1\le b< \frac{n}{2}-1$.
Denote by $N_G(v)$ or simply by $N(v)$ the set of all those vertices of $G$ that are adjacent to the vertex $v\in V (G)$.
Let $n_i(G)$ (or $n_i$) be the number of vertices of $G$ that have degree $i$.

\section{On the maximum Wiener Polarity index of chemical trees with fixed number of branching vertices}

Let $BT_{max}$ be the tree with the maximal $W_{p}$ value among all the members of $\mathcal{CT}_{n,b}^*$, where $1\leq b < \frac{n}{2}-1$. Throughout this section, we assume that $n\ge 7$. For proving the main result of this section, we firstly need some structural properties of the tree $BT_{max}$.
The following result concerning the length of the internal path of the tree $BT_{max}$ was appeared in the recent paper \cite{Noureen-20}.

\begin{lem}\label{lem-bm1} {\rm \cite{Noureen-20}}
The tree $BT_{max}\in \mathcal{CT}_{n,b}^*$ does not contain any internal path of length greater than 1.
\end{lem}

\begin{lem}\label{lem-bm2}
 If the tree $BT_{max}\in \mathcal{CT}_{n,b}^*$ contains a non-branching vertex adjacent to a vertex of degree 4, then $BT_{max}$ does not contain adjacent vertices of degree 3.
\end{lem}

\begin{proof}
Assume, on the contrary, that $w,z\in BT_{max}$ be the vertices of degree 3 such that $w$ and $z$ are adjacent to each other and there is a non-branching vertex $u$ adjacent to a vertex $v$ of degree 4. Let $w_1$ and $w_2$ be the neighbors of $w$ different from $z$, then $d_{w_1}\ge 1$ and $d_{w_2}\ge 1$. If $T^{\prime}=BT_{max}-w_{1}w-w_{2}w+uw_{1}+uw_{2}$,


then it can easily be observed that $T^{\prime}\in \mathcal{CT}_{n,b}^*$ and
\beqs
W_{p}(BT_{max})- W_{p}(T^{\prime})&=&4+2(d_{w_1}-1)+2(d_{w_2}-1)+(d_u-1)\sum_{x\in N(u)}(d_{x}-1)\\
&&-(d_u+1)(d_{w_1}-1)-(d_u+1)(d_{w_2}-1)\\
&&-(d_u+1)\sum_{x\in N(u)}(d_{x}-1)\\
&\leq&4-2\sum_{x\in N(u)}(d_{x}-1) <0,
\eeqs
which is a contradiction to the choice of $BT_{max}$.

\end{proof}

\begin{lem}\label{lem-bm3} {\rm \cite{Noureen-20}}
If the tree $BT_{max}\in \mathcal{CT}_{n,b}^*$ contains a pendent vertex adjacent to a branching vertex, then $BT_{max}$ does not contain a pendent path of length greater than 2.
\end{lem}

\begin{lem}\label{lem-bm5}
In the tree $BT_{max}\in \mathcal{CT}_{n,b}^*$ each vertex of degree 3 contains at most one neighbor of degree 4.
\end{lem}

\begin{proof}
Suppose, on the contrary, that $z$ is a vertex of degree 3 in $BT_{max}$ with neighbors $x$, $y$ such that $d_{x}=d_y=4$. Let $u$ be a non-branching vertex adjacent to a branching vertex $v$ in $BT_{max}$.


If $T^{\prime}= BT_{max}-xz-zy-uv+xy+uz+zv$, then $T^{\prime}\in \mathcal{CT}_{n,b}^*$ and
\beqs
W_{p}(BT_{max})- W_{p}(T^{\prime})&=& 8+d_{u}d_{v}-3d_{u}-3d_{v},
\eeqs
which is negative because the function $f$ defined by $f(a,b)=ab-3a-3b+8$, with $1\le a\le 2$ and $3\le b\le 4$, is negative, and hence we have
$W_{p}(BT_{max})<W_{p}(T^{\prime}),$ which is a contradiction to the choice of $BT_{max}$.
\end{proof}

\begin{lem}\label{lem-bm4}

a) If a tree $T\in \mathcal{CT}_{n,b}^*$ contains a vertex $v$ of degree 2 with neighbors $u$ and $w$ such that $d_w=3$ and $d_u\ge 1$, where $w_1$ and $w_2$ are the neighbors of $w$ different from $v$, then a tree $T^{\prime}$ can be obtained from $T$ as $T^{\prime}=T-\{ww_1,ww_2\}+\{vw_1,vw_2\}$, which gives $W_{p}(T)\le W_{p}(T^{\prime})$\,.\\

b) If a tree $T\in \mathcal{CT}_{n,b}^*$ contains vertices $u,v,w,z$ such that $uv, vw\in E(T)$ with $d_u\ge 1$, $d_v=2$, $d_w=4$ and $d_z=3$ where $z_1$, $z_2$ are neighbors of $z$ not lying on the $wz$-path, then a tree $T^{\prime}$ can be obtained from $T$ as $T^{\prime}=T-\{zz_1,zz_2\}+\{vz_1,vz_2\}$, which gives $W_{p}(T)\le W_{p}(T^{\prime})$\,.
\end{lem}

\begin{proof}
a) It holds,
\begin{align}\label{EqAA1b}
W_{p}(T)- W_{p}(T^{\prime})=&\,  2-(d_{w_1}-1)-(d_{w_2}-1)-2(d_{u}-1).
\end{align}
For $n\ge 7$ and $b\ge 1$, Eq(1) gives $W_{p}(T)- W_{p}(T^{\prime})\le 0$ or $W_{p}(T)\le W_{p}(T^{\prime})$\,.\\

b) Let $x$ be the neighbor of $z$ different from $z_1$ and $z_2$, then

\beqs
W_{p}(T)- W_{p}(T^{\prime})&=& -6-2(d_u-1)-(d_{z_1}-1)-(d_{z_2}-1)+2(d_x-1)\\
&\le& 0
\eeqs
Hence $W_{p}(T)\le W_{p}(T^{\prime})$\,.
\end{proof}

Thus, by Lemmas \ref{lem-bm1}-\ref{lem-bm4}, in order to maximize Wiener polarity index $W_{p}$ first we replace the vertices of degree 2 by the vertices of degree 4 using Lemma \ref{lem-bm4} so that the presence of vertex of degree 2 ensures that there is no vertex of degree 3, then place the vertices of degree 3 between pendent vertices and vertices of degree 4 so that if there is a pendent vertex adjacent to a vertex of degree 4, then there is no adjacent pair of vertices of degree 3, also there is no vertex of degree 3 between any two vertices of degree 4. Note that if $b=\frac{n}{2}-1$ or $n=2b+2$ the tree $BT_{max}$ contains only pendent vertices and branching vertices of degree 3 and if $b < \frac{n}{2}-1$, then there is at least one vertex of degree 4. The construction of maximal tree described above implies that when $\frac{n-2}{3}\le b< \frac{n}{2}-1$  or $2b+2< n\le 3b+2$ there is no vertex of degree 2 or $n_2=0$, $n_3=3b-n+2$ and $n_4=n-2b-2$.\\
Denote $\mathcal{BT}_{1}(n,b)\subseteq \mathcal{CT}_{n,b}^*$, for $\frac{n-2}{3}\le b< \frac{n}{2}-1,$ the set of all $n$-vertex chemical trees with degree sequence $(\underbrace{4,4,...,4}_{n-2b-2},\underbrace{3,3,...,3}_{3b-n+2},\underbrace{1,1,...,1}_{n-b})$, whose vertices of degree 3 are placed as described above.
Let $\theta$ denotes the number of non-branching neighbors of vertices of degree 4 in $\mathcal{BT}_{1}(n,b)$, or $\theta=2n_4+2=2n-4b-2$. Note that if $n_3< \theta$ or $\frac{n-2}{3}\le b< \frac{3n-4}{7}$, then there are not enough vertices of degree 3 to be placed between the pendent vertices and vertices of degree 4, in this case we get $x_{3,4}=n_3=3b-n+2$, $x_{4,4}=n_4-1=n-2b-3$, $x_{3,3}=0$. Now if $n_3 \ge \theta +1$ or $\frac{3n-4}{7} \le b < \frac{n}{2}-1$, then $x_{3,4}= \theta=2n-4b-2$, $x_{4,4}=n_4-1=n-2b-3$, $x_{3,3}=n_3-\theta=7b-3n+4$.\\

Thus, using the construction of maximal tree in $\mathcal{CT}_{n,b}^*$ as described above if $b< \frac{n-2}{3}$, then $n_{3}=0$ and $n_4=b$. By Lemmas \ref{lem-bm1}and \ref{lem-bm3}, in order to maximize $W_{p}$ we need to place the vertices of degree 2 between pendent vertices and vertices of degree 4 so that if there is any pendent vertex adjacent to a vertex of degree 4, then there is no vertex of degree 2 between two vertices of degree 4 and there is no pendent path of length greater than 2. Denote $\mathcal{BT}_{2}(n,b)\subseteq \mathcal{CT}_{n,b}^*$, for $1\leq b < \frac{n-2}{3},$ the set of all $n$-vertex chemical trees with degree sequence $(\underbrace{4,4,...,4}_{b},\underbrace{2,2,...,2}_{n-3b-2},\underbrace{1,1,...,1}_{2b+2})$, whose vertices of degree 2 are placed as described above. Let $\theta$ denotes the number of non-branching neighbors of vertices of degree 4.\\
Note that if $n_2 < \theta$ or $\frac{n-4}{5} < b < \frac{n-2}{3}$, then there are not enough vertices of degree 2 to be placed between the vertices of degree 4 and pendent vertices and $x_{4,4}=b-1$, $x_{2,4}=n_2=n-3b-2$, $x_{2,2}=0$. Now if $1\le b \le \frac{n-4}{5}$, then $x_{4,4}=b-1$, $x_{2,4}=\theta= 2b+2$, $x_{2,2}=n-5b-4$.

By previous considerations, the structure of a chemical tree that maximizing $W_{p}$ is completely determined, which enables us to state the following result.

\begin{thm}
Let $BT\in \mathcal{CT}_{n,b}^*,$ where $1\leq b \leq \frac{n-4}{2}$, then
\[
W_{P}(BT)
\leq
\begin{cases}
n+10b-7 &\text{,\ $1\leq b \le \frac{n-4}{5}$}\\
3n-15 &\text{,\ $\frac{n-4}{5}< b < \frac{3n-4}{7}$}\\
9n-14b-23 &\text{,\ $\frac{3n-4}{7} \leq b < \frac{n}{2}-1.$}\\
\end{cases}
\]
The equality holds if $BT\in \mathcal{BT}_{1}(n,b)$ for $\frac{n-2}{3}\le b< \frac{n}{2}-1$, or $BT \in \mathcal{BT}_{2}(n,b)$ for $1\leq b < \frac{n-2}{3}$.
\end{thm}
\begin{proof}
According to previous considerations, the tree from $\mathcal{BT}_{n,b}$ which maximizes $W_{p}$ belongs to $\mathcal{BT}_{1}(n,b)$ for $\frac{n-2}{3}\le b< \frac{n}{2}-1$ and $\mathcal{BT}_{2}(n,b)$ for $1\leq b < \frac{n-2}{3}$. The Wiener Polarity index of an arbitrary tree belonging to these sets can easily be calculated using above described results, which
completes the proof.

\end{proof}

\section{On the minimum Wiener Polarity index of chemical trees with fixed number of branching vertices}

In this section the influence of the number of branching vertices in a chemical tree on Wiener polarity index is investigated and the lower bound on $W_{P}$ for $n$-vertex chemical trees from $\mathcal{CT}_{n,b}^*$ containing $b$ branching vertices is found. Let $CT$ be a tree with minimum Wiener polarity index among the $n$-vertex trees from $\mathcal{CT}_{n,b}^*$. Then the following observations hold.\\

\begin{lem}\label{bn1}
The tree $CT\in \mathcal{CT}_{n,b}^*$ for $b\ge 2$ does not contain any pendent path of length greater than 1.
\end{lem}
\begin{proof}
Suppose, on the contrary, that there is a path $P:u_{0}u_{1}u_{2}...u_{t-1}u_{t}v$ with $t\geq 1$ in $CT$ where $d_v\ge 3$ and $d_{u_{0}}=1$, $d_{u_{1}}=d_{u_{2}}=...=d_{u_{t}}=2$. Let $w$ be a neighbor of  $v$ lying on some internal path (the existence of $w$ is confirmed because of the assumption $b\ge 2$).


Let $T^{\prime}=CT-\{u_{t-1}u_{t},vw\}+ \{vu_{0},u_{t-1}w\}$. It is observed that $T^{\prime}\in \mathcal{CT}_{n,b}^*$\,\,.
As $d_{w}\geq 2$ and  $d_{v}\geq 3$, we have
\beqs
W_{p}(CT)-W_{p}(T^{\prime})&=& (d_{w}-1)(d_{v}-2)> 0,
\eeqs
a contradiction to the minimality of $CT$.
\end{proof}

\begin{lem}\label{bn2}
If the tree $CT \in \mathcal{CT}_{n,b}^*$ contains an internal path of length 1, then it does not contain an internal path of length greater than 2.
\end{lem}
\begin{proof}
Suppose, on the contrary, that there is an internal path $u_1u_2\cdots u_s$ of length at least 3 in $CT$ provided that $u_1$ and $u_s$ are branching vertices, let there also exists a pair of adjacent, branching vertices $u$ and $v$ in $CT$. Let $T^{\prime}=CT-\{u_1u_2,u_2u_3,uv\}+\{u_1u_3,uu_2,u_2v\}$,
then ${T^{\prime}}\in \mathcal{CT}_{n,b}^*$ and
\beqs
W_{p}(CT)-W_{p}(T^{\prime})&=& d_{u}d_{v}-2d_{u}-2d_{v}+4,
\eeqs
which is positive because the function $f$ defined by $f(x,y)=xy-2x-2y+4$, with $3\le x,y\le4$, is increasing in both $x$ and $y$, and hence we have $W_{p}(CT)>W_{p}(T^{\prime})$, a contradiction to the choice of $CT$
\end{proof}

\begin{lem}\label{bn3}
Let $u$ be a vertex of degree 4 in a tree $T\in \mathcal{CT}_{n,b}^*$ with a non-pendent neighbor $u_1$, and a neighbor $u_2$ (different from $u_1$) such that a pendent vertex $w_1$ is connected to $u$ via $u_2$ and $N(w_1)=w_2$. Let $T^{\prime}$ be a tree obtained by deleting the edge $uu_1$ and adding a new edge $w_1u_1$(Fig). Then $T^{\prime}\in \mathcal{CT}_{n,b}^*$ and $W_{p}(T) \ge W_{p}(T^{\prime})$.


\end{lem}

\begin{proof}
It holds as if $u_2\neq w_1$, then we obtain
\beqs
W_{p}(T)-W_{p}(T^{\prime})&=& 2(d_{u_1}-1)+\sum_{x\in N(u), x\neq u_1}{(d_x-1)}-(d_{w_2}-1)\\
&\ge& 0,
\eeqs
which implies that $W_{p}(T) \ge W_{p}(T^{\prime})$.\\
Now if $u_2 = w_1$ i.e. $u=w_2$, then we obtain
\beqs
W_{p}(T)-W_{p}(T^{\prime})&=& -2+2(d_{u_1}-1)+\sum_{x\in N(u), x\neq u_1}{(d_x-1)}\\
&\ge& 0,
\eeqs
again implying that $W_{p}(T) \ge W_{p}(T^{\prime})$.
\end{proof}

\begin{remark}
Bearing in mind Lemma \ref{bn3}, it can be observed that by inserting the vertices of degree 2, all the vertices of degree 4 can be replaced by the vertices of degree 3 without loss of minimality of $W_p$ of the trees from $\mathcal{CT}_{n,b}^*$.
\end{remark}

Consequently, using Lemmas \ref{bn1},\ref{bn2} and \ref{bn3} a tree $CT\in \mathcal{CT}_{n,b}^*$ containing only vertices of degree 1, 2 or 3 can be constructed with $n_{4}=0$, $n_{3}=b$, $n_{2}=n-2b-2$ and $n_{1}=b+2$ where the vertices of degree 2 are placed between two vertices of degree 3 so that there is no vertex of degree 2 adjacent to a pendent vertex and if there is an internal path of length 1 then there is no internal path of length greater than 2, the remaining vertices of degree 2 are placed arbitrarily between any two vertices of degree 2.\\
Note that if $n_2< b-1$ or $n< 3b+1$, then there are not enough vertices of degree 2 to place between the vertices of degree 3 so we get $x_{2,3}=2n_2=2n-4b-4$, $x_{3,3}=b-1-n_2$ and $x_{2,2}=0$. If $n_2\ge b-1$ or $n\ge 3b+1$, then $x_{2,3}=2(n_3-1)=2b-2$, $x_{3,3}=0$ and $x_{2,2}=n_2-(b-1)=n-3b-1$. Hence, the following result can be concluded
\begin{thm}
Let $CT\in \mathcal{CT}_{n,b}^*$, where $1\leq b \leq \frac{n}{2}-1,$ then
\[
W_{P}(CT)
\geq
\begin{cases}
b+n-5 &\text{,$n\geq 3b+1$}\\
4b-4  &\text{,$n<3b+1$.}
\end{cases}
\]
and equality holds if $CT\in \mathcal{B}_{n,b}$, where $\mathcal{B}_{n,b}$ is the set of $n$-vertex trees with the degree sequence $(\underbrace{3,3,...,3}_{b},\underbrace{2,2,...,2}_{n-2b-2},\underbrace{1,1,...,1}_{b+2})$.
\end{thm}

\begin{proof}
 By previous considerations, the tree from $\mathcal{CT}_{n,b}^*$ which gives a lower bound to $W_{p}$ belongs to $\mathcal{B}_{n,b}$ and the Wiener polarity index of an arbitrary tree from $\mathcal{B}_{n,b}$ can easily be calculated, which completes the proof.
\end{proof}

\section{On the maximum Wiener polarity index of chemical trees with given number of segments}

For $n\ge6$ and $3\le k\le n-1$, denote by $\mathcal{CT}_{n,k}$ by the class of all $n$-vertex chemical trees with $k$ segments.
We determine the structure of the chemical tree(s) from the class $\mathcal{CT}_{n,k}$ that maximizes the Wiener polarity index $W_{P}$. We denote this extremal chemical tree by $CT_{max}.$

\begin{lem}\label{sm1} {\rm \cite{Noureen-20}}
The tree $CT_{max}\in \mathcal{CT}_{n,k}$ does not contain any internal path of length greater than 1.
\end{lem}

\begin{lem}\label{sm2} {\rm \cite{Noureen-20}}
If the tree $CT_{max}\in \mathcal{CT}_{n,k}$ contains a pendent vertex adjacent to a branching vertex, then $CT_{max}$ does not contain a pendent path of length greater than 2.
\end{lem}

\begin{lem}\label{sm3} {\rm \cite{Noureen-20}}
If the tree $CT_{max}\in \mathcal{CT}_{n,k}$ contains a pendent vertex adjacent to a vertex of degree 4 then $CT_{max}$ does not contain any vertex of degree 3 adjacent to a vertex of degree 2.

\end{lem}

\begin{lem}\label{sm3b}
Every path in the tree $CT_{max}\in \mathcal{CT}_{n,k}$ contains at most two vertices of degree 3.

\end{lem}

\begin{proof}

Suppose, on the contrary, that the tree $CT_{max}$ has a path containing three vertices of degree 3.
We assume that $u$-$w$ is a path in $CT_{max}$ containing exactly three vertices of degree 3, where $u$ and $w$ have degree 3. Let $v$ be the third vertex of degree 3 lying on the $u$-$w$ path and $w_1,w_2$ be the neighbors of $w$ which do not lie on the path $u$-$w$. Let $CT^{\prime}=CT_{max}-\{ww_1,ww_2\} + \{uw_1,vw_2\}$, then
\begin{align}\label{Eq-AA1}
W_{p}(CT_{max})-W_{p}(CT^{\prime})=&  -\sum_{x\in N(u)}(d_{x}-1) - \sum_{y\in N(v)}(d_{y}-1)+ 2\sum_{z\in N(w), \ z\neq w_1, \ z\neq w_2}(d_{z}-1)\nonumber\\
&-(d_{w_{1}}-1)-(d_{w_{2}}-1)
\end{align}

If $vw\in E(CT_{max})$ then by Lemma \ref{sm1}, the vertex $u$ has at least one branching neighbor and the vertex $v$ has at least two branching neighbors and hence from \eqref{Eq-AA1}, it follows that $W_{p}(CT_{max})-W_{p}(CT^{\prime})<0$, which is a contradiction to the definition of $CT_{max}$.

If $vw\not\in E(CT_{max})$ then by Lemma \ref{sm1}, the vertex $u$ has at least one branching neighbor and the vertex $v$ has at least two branching neighbors, one of which must have degree 4, and hence again from \eqref{Eq-AA1}, a contradiction is obtained.

\end{proof}

\begin{lem}\label{sm4}
If $T$ is a tree containing more than two vertices of degree 3, such that every path in $T$ contains at most two vertices of degree 3, and if $T$ has the maximum Wiener polarity index $W_{P}$ among all the members of the class $\mathcal{CT}_{n,k}$ then there exists a tree $T'\in \mathcal{CT}_{n,k}$ containing at most two vertices of degree 3 such that $W_{p}(T) \le W_{p}(T')$.
\end{lem}

\begin{proof}
Let $u,v,w\in V(T)$ such that $d_{u}=d_{v}=d_{w}=3$. Let $w_1$ and $w_2$ be those neighbors of $w$ which lie on neither of the paths $u$-$w$ and $v$-$w$. If $T'=T-\{ww_1,ww_2\} + \{uw_1,vw_2\}$, then we have an equation similar to \eqref{Eq-AA1}:
\begin{align}\label{Eq-AA1a}
W_{p}(T)-W_{p}(T')=&  -\sum_{x\in N(u)}(d_{x}-1) - \sum_{y\in N(v)}(d_{y}-1)+ 6
-(d_{w_{1}}-1)-(d_{w_{2}}-1)
\end{align}
Lemma \ref{sm1} guaranties that every vertex on the paths $u$-$w$ and $v$-$w$, except the vertices $u,v,w$, has degree 4 and hence from \eqref{Eq-AA1a}, it follows that $W_{p}(T)-W_{p}(T')\le0$.

\end{proof}

\begin{lem}\label{sm3-new}
If the tree $CT_{max}\in \mathcal{CT}_{n,k}$ contains a vertex $u$ of degree 3 then $u$ does not have more than one branching neighbor.

\end{lem}

\begin{proof}
Suppose, on the contrary, that $v$ and $w$ are two branching neighbors of $u$. Let $P=v_1v_2\cdots v_{i-1}v_iv_{i+1}\cdots v_r$ be the longest path containing $u$, $v$ and $w$, where $v_{i-1}=v$, $v_{i}=u$ and $v_{i+1}=w$.
By Lemma \ref{sm3b}, $P$ contains at most two vertices of degree 3, including $u$. If $P$ has two vertices of degree 3 including $u$ then, without loss of generality, we assume that $d_{v_j}=3$ for some $j$, $1\le j \le i-1$. Thus, there exists some $k$ with $i+1\le k \le r-1$ such that $v_k$ has exactly one branching neighbor and $d_{v_k}=4$. If $CT^{\prime}=CT_{max}-\{v_{i-1}v_i,v_{i}v_{i+1},v_{k}v_{k+1}\} + \{v_{i-1}v_{i+1},v_{k}v_{i},v_{i}v_{k+1}\}$ then bearing in mind the facts $d_{v_{k+1}}\le 2$, $d_{v_{i+1}}=4$ and $d_{v_{i-1}}=3$ or 4, we have
\begin{align*}
W_{p}(CT_{max})-W_{p}(CT^{\prime})=&  -d_{v_{i-1}} + d_{v_{k+1}} < 0\,,
\end{align*}
a contradiction to definition of the tree $CT_{max}$.

\end{proof}

Keeping in mind Lemmas \ref{sm1} and \ref{sm2}, we define a class $\mathcal{CT}_{1}(n,k)$
consisting of the trees with the degree sequence $(\underbrace{4,4,...,4}_{\frac{k-1}{3}},\underbrace{2,2,...,2}_{n-k-1},\underbrace{1,1,...,1}_{\frac{2k+4}{3}})$ such that every internal path has length 1 and if there is a starlike pendent edge then there is no pendent path of length greater than 2. Clearly, it holds $k\equiv$ 1 (mod 3) for such trees. Also, it is clear that if $n_1> n_2$ for such trees then $n < \frac{5k+7}{3}$ and it holds $x_{{1,2}}=x_{2,4}=n_2=n-k-1$, $x_{1,4}=n_1-n_2=\frac{5k-3n+7}{3}$, $x_{2,2}=0$ and $x_{4,4}=\frac{k-4}{3}$. If $n_1\le n_2$, then $n \ge \frac{5k+7}{3}$ and it holds $x_{{1,2}}=x_{2,4}=n_1=\frac{2k+4}{3}$, $x_{1,4}=0$, $x_{2,2}=n_2-n_1=\frac{3n-5k-7}{3}$ and $x_{4,4}=\frac{k-4}{3}$.

Now, bearing in mind Lemmas \ref{sm1}, \ref{sm2}, \ref{sm3} and \ref{sm3-new} we define a new class $\mathcal{CT}_{2}(n,k)$
consisting of the trees with the degree sequence $(\underbrace{4,4,...,4}_{\frac{k-3}{3}},3,\underbrace{2,2,...,2}_{n-k-1},\underbrace{1,1,...,1}_{\frac{2k+3}{3}})$ and satisfying the following properties\\
$\bullet$ every internal path has length 1,\\
$\bullet$ if there is a starlike pendent edge then there is no pendent path of length greater than 2,\\
$\bullet$ the vertex of degree 3 does not have more than one branching neighbor,\\
$\bullet$ if there is a pendent neighbor of a vertex of degree 4 then the vertex of degree 3 does not have any neighbor of degree 2.\\
Clearly, it holds $k\equiv$ 0 (mod 3) for the trees of $\mathcal{CT}_{2}(n,k)$. Denote by $\Theta$ the number of non-branching neighbors of the vertices of degree 4 in a tree $T \in\mathcal{CT}_{2}(n,k)$. It is clear $\Theta=2n_4 +1=\frac{2k-3}{3}$, for every tree in $\mathcal{CT}_{2}(n,k)$. If $n_2\le \Theta=\frac{2k-3}{3}$, that is, if $n\le \frac{5k}{3}$ then for the trees of $\mathcal{CT}_{2}(n,k)$ it holds $x_{2,2}=x_{2,3}=x_{3,3}=0$, $x_{2,4}=n_2=n-k-1$, $x_{3,4}=1$, $x_{4,4}=n_4-1=\frac{k-6}{3}$ and hence $W_p(T)=3n-15$ for every $T\in\mathcal{CT}_{2}(n,k)$ when $n\le \frac{5k}{3}$. Now, if $n_2= \Theta+1=\frac{2k-3}{3}+1$, that is, if $n= \frac{5k}{3}+1$ then for the trees of $\mathcal{CT}_{2}(n,k)$ we have $x_{2,2}=x_{3,3}=0$, $x_{2,3}=1$, $x_{2,4}=n_2-1=n-k-2$, $x_{3,4}=1$, $x_{4,4}=n_4-1=\frac{k-6}{3}$ and hence $W_p(T)=3n-16$ for every $T\in\mathcal{CT}_{2}(n,k)$ when $n= \frac{5k}{3}+1$. Finally, if $n_2> \Theta+1=\frac{2k-3}{3}+1$, that is, if $n> \frac{5k}{3}+1$ then for the trees of $\mathcal{CT}_{2}(n,k)$ we have $x_{2,2}=n_2-\Theta-2=\frac{3n-5k-6}{3}$, $x_{3,3}=0$, $x_{2,3}=2$, $x_{2,4}=\Theta=\frac{2k-3}{3}$, $x_{3,4}=1$, $x_{4,4}=n_4-1=\frac{k-6}{3}$ and hence $W_p(T)=\frac{3n+10k-39}{3}$ for every $T\in\mathcal{CT}_{2}(n,k)$ when $n> \frac{5k}{3}+1$.

Lastly, bearing in mind Lemmas \ref{sm1}, \ref{sm2}, \ref{sm3}, \ref{sm4} and \ref{sm3-new} we define another class $\mathcal{CT}_{3}(n,k)$
consisting of the trees with the degree sequence $(\underbrace{4,4,...,4}_{\frac{k-5}{3}},3,3,\underbrace{2,2,...,2}_{n-k-1},\underbrace{1,1,...,1}_{\frac{2k+2}{3}})$ and satisfying the following properties\\
$\bullet$ every internal path has length 1,\\
$\bullet$ if there is a starlike pendent edge then there is no pendent path of length greater than 2,\\
$\bullet$ every vertex of degree 3 has exactly one branching neighbor,\\
$\bullet$ if there is a pendent neighbor of a vertex of degree 4 then every vertex of degree 3 has exactly two pendent neighbors.\\
Clearly, it holds $k\equiv$ 2 (mod 3) for the trees of $\mathcal{CT}_{3}(n,k)$. If $k=5$, then $n_4=0$, for this particular value of $k$, if $n-k-1< 4$ that is, if $n<k+5$, then for the trees of $\mathcal{CT}_{3}(n,k)$ it holds $x_{1,2}=x_{2,3}=n-k-1$, $x_{2,2}=0$, $x_{1,3}=k+5-n$ and $x_{3,3}=1$ and hence $W_P(T)=2n-2k+2$ for every $T\in\mathcal{CT}_{3}(n,k)$ when $n<k+5$ . If $n\ge k+5$, then then for the trees of $\mathcal{CT}_{3}(n,k)$ it holds $x_{1,2}=x_{2,3}=4$, $x_{2,2}=n-k-5$, $x_{1,3}=0$ and $x_{3,3}=1$ and hence $W_P(T)=n-k+7$ for every $T\in\mathcal{CT}_{3}(n,k)$ when $n\ge k+5$ . Denote by $\Theta^{\prime}$ the number of non-branching neighbors of the vertices of degree 4 in a tree $T \in\mathcal{CT}_{3}(n,k)$. It is clear $\Theta^{\prime}=2n_4=2(\frac{k-5}{3})$, for every tree in $\mathcal{CT}_{3}(n,k)$. If $n_2\le \Theta^{\prime}$, that is, if $ n\le \frac{5k-7}{3}$, then for the trees of $\mathcal{CT}_{3}(n,k)$ it holds $x_{2,2}=x_{2,3}=x_{3,3}=0$, $x_{2,4}=n_2=n-k-1$, $x_{3,4}=2$, $x_{4,4}=n_4-1=\frac{k-8}{3}$ and hence $W_p(T)=3n-15$ for every $T\in\mathcal{CT}_{3}(n,k)$ when $ n\le \frac{5k-7}{3}$. Now if $\Theta^{\prime}+1\le n_2\le \Theta^{\prime}+3$ that is, $\frac{5k-4}{3}\le n\le \frac{5k+2}{3}$ then for the trees of $\mathcal{CT}_{3}(n,k)$ we have $x_{2,2}=x_{3,3}=0$, $x_{2,3}= n_2-\Theta^{\prime}=\frac{3n-5k+7}{3}$, $x_{2,4}=\Theta^{\prime}=2(\frac{k-5}{3})$, $x_{3,4}=2$, $x_{4,4}=n_4-1=\frac{k-8}{3}$ and hence $W_p(T)=\frac{6n+5k-52}{3}$ for every $T\in\mathcal{CT}_{3}(n,k)$ when $\frac{5k-4}{3}\le n\le \frac{5k+2}{3}$\,. Finally, if $n_2>\Theta^{\prime}+3$ that is, $n> \frac{5k+2}{3}$ then for the trees of $\mathcal{CT}_{3}(n,k)$ we have $x_{2,2}= n_2-\Theta^{\prime}-4=\frac{3n-5k-5}{3}$, $x_{3,3}=0$, $x_{2,3}= 4$, $x_{2,4}=\Theta^{\prime}=2(\frac{k-5}{3})$, $x_{3,4}=2$, $x_{4,4}=n_4-1=\frac{k-8}{3}$ and hence $W_p(T)=\frac{3n+10k-47}{3}$ for every $T\in\mathcal{CT}_{3}(n,k)$ when $n > \frac{5k+2}{3}$\,.

By previous considerations, the structure of a chemical tree that maximizes $W_{p}$ is completely determined, which enables us to state the following result.

\begin{thm}
Let $CT\in \mathcal{CT}_{n,k},$ where $3\leq k \leq n-1$, then
\[
W_{P}(CT)
\leq
\begin{cases}
3n-15 &\text{,\ $n \le \frac{5k}{3},$ \ $k\equiv 0(mod3)$ }\\
3n-16 &\text{,\ $ n = \frac{5k}{3}+1,$ \ $k\equiv 0(mod3)$}\\
\frac{3n+10k-39}{3}&\text{,\ $n> \frac{5k}{3}+1,$ \ $k\equiv 0(mod3)$}\\
3n-15&\text{,\ $n < \frac{5k+7}{3},$ \ $k\equiv 1(mod3)$ }\\
\frac{3n+10k-31}{3}&\text{,\ $n\geq \frac{5k+7}{3},$ \ $k\equiv 1(mod3)$}\\
3n-15&\text{,\ $n\le \frac{5k-7}{3},$ \ $k\equiv 2(mod3)$}\\
\frac{6n+5k-52}{3}&\text{,\ $\frac{5k-4}{3}\leq n \leq \frac{5k+2}{3},$ \ $k\equiv 2(mod3)$}\\
\frac{3n+10k-47}{3}&\text{,\ $ n > \frac{5k+2}{3},$ \ $k\equiv 2(mod3).$}
\end{cases}
\]
The equality holds if $CT\in \mathcal{CT}_{2}(n,k)$ for $k\equiv 0(mod3),$ $CT \in \mathcal{CT}_{1}(n,k)$ for $k\equiv 1(mod3)$, or $CT \in \mathcal{CT}_{3}(n,k)$ for $k\equiv 2(mod3)$.
\end{thm}

\begin{proof}
Using Lemmas \ref{sm1}-\ref{sm3-new} and previous considerations, it can be concluded that the tree that maximizes $W_{p}$ belongs to $\mathcal{CT}_{2}(n,k)$ for $k\equiv 0(mod3),$ $\mathcal{CT}_{1}(n,k)$ for $k\equiv 1(mod3)$ and $\mathcal{CT}_{3}(n,k)$ for $k\equiv 2(mod3)$. Wiener polarity indices of these trees belonging to these sets can easily be calculated by using simple calculations, which completes the proof.
\end{proof}

\section{On the minimum Wiener polarity index of chemical trees with given number of segments}
 Denote $\mathcal{CT}_{n,k}^{*} \subseteq \mathcal{CT}_{n,k}$ the class of $n$-vertex chemical trees with $k$ segments. Let $CT_{min}$ be the tree that minimizes the Wiener polarity index in the class $\mathcal{CT}_{n,k}^{*}$ for $7\leq k \leq n-2$.

\begin{lem}\label{sn1}
The tree $CT_{min}$ does not contain any pendent path of length greater than 1.
\end{lem}

\begin{proof}
Suppose, on the contrary, that there is a path $P:u_{0}u_{1}u_{2}...u_{t-1}u_{t}v$ with $t\geq 1$ in $CT_{min}$ where $d_v\ge 3$ and $d_{u_{0}}=1$, $d_{u_{1}}=d_{u_{2}}=...=d_{u_{t}}=2$. Let $w$ be a neighbor of  $v$ lying on some internal path (the existence of $w$ is confirmed because of the assumption $k\ge 7$). Let $CT^{\prime}=CT_{min}-\{u_{t-1}u_{t},vw\}+ \{vu_{0},u_{t-1}w\}$. It is observed that $CT^{\prime}\in \mathcal{CT}_{n,k}^{*}$\,\,.
As $d_{w}\geq 2$ and  $d_{v}\geq 3$, we have
\beqs
W_{p}(CT_{min})-W_{p}(CT^{\prime})&=&(d_{v}-1)(d_{w}-1)-(d_{w}-1)\\
&=&(d_{w}-1)(d_{v}-2)> 0,
\eeqs
a contradiction to the minimality of $CT_{min}$.
\end{proof}

\begin{lem}\label{sn2}
The tree $CT_{min}$ contains at least one vertex of degree 4.
\end{lem}

\begin{proof}
Assume, on the contrary, the maximum degree of $CT_{min}$ is 3.
Let $P: v_1v_2\cdots v_r$ be the longest path in $CT_{min}$. The supposition $k\geq 7$ ensures that $P$ contains at least 3 vertices of degree 3. We note that both the vertices $v_1$, $v_r$ are pendent and by using Lemma \ref{sn1} we deduce that both the vertices $v_2$, $v_{r-1}$ have degree 3. If $P$ does not contain any branching vertex, different from $v_2$ and $v_{r-1}$, having a pendent neighbor. Let $u\not\in\{v_2,v_{r-1}\}$ be a branching vertex on $P$ with the non-pendent neighbor $u_1$, not lying on the path $P$. If $T'$ is the tree obtained from $CT_{min}$ by removing the neighbor(s) of $u_1$ different from $u$ and adding the these neighbor(s) to the vertex $v_r$, then $W_p(CT_{min})=W_p(T')$. Hence, we may assume that $P$ contains at least one branching vertex, say $v_i$, different from $v_2$ and $v_{r-1}$, having a pendent neighbor $w$. As $k\leq n-2$, so there is at least one vertex of degree 2 in $CT_{min}$. Without loss of generality we may assume that $d_{v_{r-2}}=2$ and $d_{v_{3}}\leq d_{v_{i-1}}\leq d_{v_{i+1}}$. Let
$CT^{\prime}=CT_{min}-\{wv_{i},v_{i-1}v_{i},v_{i}v_{i+1}\}+\{v_{i-1}v_{i+1}+v_{i}v_{r-1}+v_{2}w$\}.


then $CT^{\prime}\in \mathcal{CT}_{n,k}^{*}$ and
\beqs
W_{p}(CT_{min})-W_{p}(CT^{\prime})&=& 2(d_{v_{i-1}}-1)+2(d_{v_{i+1}}-1)+2+2(d_{v_{3}}-1)\\
&&- 3(d_{v_{3}}-1)-(d_{v_{i-1}}-1)(d_{v_{i+1}}-1)-3\\
&=&-5-d_{v_{3}}+3d_{v_{i-1}}+3d_{v_{i+1}}-d_{v_{i-1}}d_{v_{i+1}}
\eeqs
As, $d_{v_{3}}\leq d_{v_{i-1}}\leq d_{v_{i+1}}$\\
Note that $I=3d_{v_{i-1}}+3d_{v_{i+1}}-d_{v_{i-1}}d_{v_{i+1}}-d_{v_{3}}$, implying that:\\
i) $I=6$, for $d_{v_{3}}=d_{v_{i-1}}= d_{v_{i+1}}$\\
ii) $I=7$, for $2=d_{v_{3}}\leq d_{v_{i-1}}< d_{v_{i+1}}$\\
iii) $I=7$, for $2=d_{v_{3}}< d_{v_{i-1}}= d_{v_{i+1}}.$\\\\
Therefore, in each possible case we get $W_{p}(CT_{min})-W_{p}(CT^{\prime})> 0$, implying that, $W_{p}(CT_{min})>W_{p}(CT^{\prime})$, a contradiction to the choice of $CT_{min}$.
\end{proof}

\begin{lem}\label{sn3}
If the tree $CT_{min}\in \mathcal{CT}_{n,k}^{*}$ contains an internal path of length 1, then it does not contain an internal path of length greater than 2.
\end{lem}
\begin{proof}
Suppose, on the contrary, that there is an internal path $u_1u_2\cdots u_s$ of length at least 3 in $CT_{min}$ provided that $u_1$ and $u_s$ are branching vertices, let there also exists a pair of adjacent, branching vertices $u$ and $v$ in $CT_{min}$. Let $CT^{\prime}=CT_{min}-\{u_1u_2,u_2u_3,uv\}+\{u_1u_3,uu_2,u_2v\}$,
then ${CT^{\prime}}\in \mathcal{CT}_{n,k}^{*}$ and
\beqs
W_{p}(CT_{min})-W_{p}(CT^{\prime})&=& d_{u}d_{v}-2d_{u}-2d_{v}+4,
\eeqs
which is positive because the function $f$ defined by $f(x,y)=xy-2x-2y+4$, with $3\le x,y\le4$, is increasing in both $x$ and $y$, and hence we have $W_{p}(CT_{min})>W_{p}(CT^{\prime})$, a contradiction to the choice of $CT_{min}$
\end{proof}

\begin{lem}\label{sn4}
If $CT_{min}\in \mathcal{CT}_{n,k}^{*}$ contains a vertex of degree 2 with the non pendent neighbors $x$ and $y$ such that $4< d_{x}+d_{y}<8,$ then $CT_{min}$ does not contain any pair of adjacent branching vertices $x^{\prime}$ and $y^{\prime}$ with $d_{x^{\prime}}+d_{y^{\prime}}>d_{x}+d_{y}.$
\end{lem}

\begin{proof}
Suppose, on the contrary, that there is an adjacent pair of branching vertices $x^{\prime}$ and $y^{\prime}$ with $d_{x^{\prime}}+d_{y^{\prime}}>d_{x}+d_{y}$ in $CT_{min}$\,, where $x$ and $y$ are the non-pendent neighbors of a vertex $z$ of degree 2 in $CT_{min}$ such that $4< d_{x}+d_{y}<8$.


Let a tree $CT^{\prime}$ is obtained as follows:
$$CT^{\prime}=CT_{min}-\{xz,zy,x^{\prime}y^{\prime}\}+\{xy,x^{\prime}z,zy^{\prime}\}$$
It can be observed that ${CT^{\prime}}\in $ and
\beqs
W_{p}(CT_{min})-W_{p}(CT^{\prime})&=&2(d_{x}+d_{y})-2(d_{x^{\prime}}+d_{y^{\prime}})+d_{x^{\prime}}d_{y^{\prime}}-d_{x}d_{y}\\
&=&(d_{x^{\prime}}d_{y^{\prime}}-2(d_{x^{\prime}}+d_{y^{\prime}}))-(d_{x}d_{y}-2(d_{x}+d_{y}))\\
&>&0,
\eeqs
a contradiction.
\end{proof}
By previous considerations we can state the following result.
\begin{thm}
If $T$ is a chemical tree with minimum Wiener Polarity index $W_{p}$ for $7\leq k \leq n-2$, then $T\in \mathcal{CT}_{n,k}^{*}.$
\end{thm}

We end this article by noting that Theorems 2.6  and 4.7 gives the complete solution of the problem of finding the chemical trees of order $n$ with a fixed number of segments or branching vertices and having the maximal $W_p$ value, which was left open in the recent paper \cite{Noureen-20}.


\begin{thebibliography}{00}

\bibitem{Ali-18a} A. Ali, Z. Du, M. Ali, A note on chemical trees with minimum Wiener polarity index, Appl. Math. Comput. 335 (2018) 231--236.
\bibitem{Ashrafi-17} A. R. Ashrafi, A. Ghalavand, Ordering chemical trees by Wiener polarity index, Appl. Math. Comput. 313 (2017) 301--312.
\bibitem{Balban-13} A. T. Balaban, Chemical graph theory and the sherlock holmes principle, HYLE 19 (2013) 107--134.
\bibitem{Chen-16} L. Chen, T. Li, J. Liu, Y. Shi, H. Wang, On the Wiener polarity indexof lattice networks, PLoS One 11 (2016) e0167075.
\bibitem{Deng-11} H. Deng, On the extremal Wiener polarity index of chemical trees, MATCH Commun. Math. Comput. Chem. 66 (2011) 305--314.
\bibitem{Du-18a}Z. Du, A. Ali, The Alkanes with maximum Wiener polarity index, Mol. Inf. 37 (2018) 1800076.
\bibitem{Du-18b}Z. Du, A. Ali, The inverse Wiener polarity index problem for chemical trees, . PLoS ONE 13 (2018) e0197142.
\bibitem{Dobrynin} A. A. Dobrynin, R. Entringer, I. Gutman, Wiener index of trees: theory and applications, Acta Appl. Math. 66 (2001) 211--249.

\bibitem{Du-09} W. Du, X. Li, Y. Shi, Algorithms and extremal problem on Wiener polarity index, MATCH Commun. Math. Comput. Chem. 62 (2009) 235--244.

\bibitem{Deng-10a} H. Deng, H. Xiao, The maximum Wiener polarity index of trees with $k$ pendants, Appl. Math. Lett. 23 (2010) 710--715.


\bibitem{Deng-10} H. Deng, H. Xiao, F. Tang, On the extremal Wiener polarity index of trees with a given diameter, MATCH Commun. Math. Comput. Chem. 63 (2010) 257--264.

\bibitem{Estrada-13}  E. Estrada, D. Bonchev, Chemical Graph Theory, in Handbook of Graph Theory, 2nd ed., J. L. Gross, J. Yellen, P. Zhang (Eds.), CRC Press, Boca Raton, FL, 2013, pp. 1538--1558.
\bibitem{Furtula-14}B. Furtula, I. Gutman, S. Ediz, On the difference of Zagreb indices, Discrete Appl. Math. 178 (2014) 83--88.

\bibitem{Gutman-14} I. Gutman, B. Furtula, C. Elphick, Three new/old vertex-degree-based topological indiceds, MATCH Commun. Math. Comput. Chem. 72 (2014) 617--632.


\bibitem{Hua-16} H. Hua, K. C. Das, On the Wiener polarity index of graphs, Appl. Math. Comput. 280 (2016) 162--167.

\bibitem{Hosoya-02}H. Hosoya, Y. Gao, in: D. H. Rouvray, R. B. King (Eds.), Topology in Chemistry – Discrete Mathematics of Molecules, Horwood, Chichester 2002 pp. 38--57.

\bibitem{Hou-12} H. Hou, B. Liu, Y. Huang, The maximum Wiener polarity index of unicyclic graphs, Appl. Math. Comput. 218 (2012) 10149--10157.

\bibitem{Lin-14} H. Lin, On the Wiener index of trees with given number of branching vertices, MATCH Commun. Math. Comput. Chem. 72 (2014) 301--310.

\bibitem{Luk-98} I. Lukovits, W. Linert, Polarity–numbers of cycle–containing structures, J. Chem. Inf. Comput. Sci. 38 (1998) 715--719.

\bibitem{Liu-12} M. Liu, B. Liu, The second Zagreb indices and Wiener polarity indices of trees wit given degree sequences, MATCH Commun. Math. Comput. Chem. 67 (2012) 439--450.

\bibitem{Lei-18} H. Lei, T. Li, Y. Ma, H. Wang, Analyzing lattice networks through substrucutres, Appl. Math. Comput. 329 (2018) 297--314.

\bibitem{Lei-17} H. Lei, T. Li, Y. Shi, H. Wang, Wiener polarity index and its generalization in trees, MATCH Commun. Math. Comput. Chem. 78 (2017) 199--212.

\bibitem{Milic-04} A. Mili$\check{c}$evi$\acute{c}$, S. Nikoli$\acute{c}$, On variable Zagreb indices, Croat. Chem. Acta 77 (2004) 97--101.

\bibitem{Ma-16} J. Ma, Y. Shi, Z. Wang, J. Yue, On Wiener polarity index of bicyclic networks, Sci. Rep. 6 (2016) 19066.

\bibitem{Ma-14} J. Ma, Y. Shi, J. Yue,The Wiener polarity index of graph products,  Ars Combin. 116 (2014) 235--244.

\bibitem{Noureen-20} S. Noureen, A. Ali, A. A. Bhatti, On the extremal Zagreb indices of $n$-vertex chemical trees with fixed number of segments or branching vertices, MATCH Commun. Math. Comput. Chem. 84 (2020) 513--534.


\bibitem{Shafi-17a} F. Shafiei, Prediction of physical and thermodynamic properties of aliphatic ethers from molecular structures by multiple linear regression, J. Chil. Chem. Soc. 62 (2017) 3389--3392.

\bibitem{Shaf-17} S. Shafique, A. Ali, On the reduced second Zagreb index of trees, Asian-European J. Math. 10 (2017) 1750084.

\bibitem{Shafi-17b} F. Shafi, A. Saeidifar, QSPR study of some physicochemical properties of sulfonamides using topological and quantum chemical indices, J. Chem. Soc. Pak. 39 (2017) 366--373.
\bibitem{Safari-17} A. Safari, F. Shafiei, A. Saeidifar, QSPR models of physicochemical properties of natural amino acids by using topological indices and MLR method, J. Chem. Soc. Pak. 39 (2017) 752--757.

\bibitem{Tratnik-19} N. Tratnik, Formula for calculating the Wiener polarity index with applications to benzenoid graphs and phenylenes, J. Math. Chem. 57 (2019) 370--383.
\bibitem{Wiener-47} H. Wiener, Structural determination of paraffin boiling points, J. Am. Chem. Soc. 69 (1947) 17--20.
\bibitem{Yue-18} J. Yue, H. Lei, Y. Shi, On the generalized Wiener polarity index of trees with a given diameter, Discrete Appl. Math. 243 (2018) 279--285.

\bibitem{Zhang-16} Y. Zhang, Y. Hu, The Nordhaus-Gaddum-type inequality for the Wiener polarity index, Appl. Math. Comput. 273 (2016) 880--884.
\end{thebibliography}
\end{document}